\documentclass[12pt]{amsart}
\usepackage{amsmath,amssymb,enumerate}
\newcommand{\IB}{{\mathbb B}}
\newcommand{\IN}{{\mathbb N}}
\newcommand{\IQ}{{\mathbb Q}}

\newcommand{\IR}{{\mathbb R}}
\newcommand{\her}{\mathrm{h}}
\DeclareMathOperator{\interior}{int}

\newcommand{\ip}[1]{\mathopen{\langle}#1\mathclose{\rangle}}
\newtheorem*{thm*}{Main Theorem}
\newtheorem{thm}{Theorem}
\newtheorem{lem}[thm]{Lemma}
\newtheorem{prop}[thm]{Proposition}
\newtheorem*{cor}{Corollary}
\theoremstyle{definition}
\newtheorem{exam}[thm]{Example}
\title[Kazhdan's property (T)]{Noncommutative real algebraic geometry of Kazhdan's property (T)}
\author{Narutaka Ozawa}
\address{RIMS, Kyoto University, \mbox{606-8502} Japan}
\email{narutaka@kurims.kyoto-u.ac.jp}
\thanks{The author was partially supported by JSPS (23540233)}

\date{2014/01/08}

\begin{document}
\begin{abstract}
It is well-known that a finitely generated group $\Gamma$ has Kazhdan's property (T)
if and only if the Laplacian element $\Delta$ in ${\mathbb R}[\Gamma]$ has a spectral gap.
In this paper, we prove that this phenomenon is witnessed in ${\mathbb R}[\Gamma]$.
Namely, $\Gamma$ has property (T) if and only if there are a constant $\kappa>0$ and
a finite sequence $\xi_1,\ldots,\xi_n$ in ${\mathbb R}[\Gamma]$ such that
$\Delta^2-\kappa\Delta = \sum_i \xi_i^*\xi_i$. This result suggests the possibility of
finding new examples of property (T) groups by solving equations in ${\mathbb R}[\Gamma]$,
possibly with an assist of computers.
\end{abstract}
\maketitle
\section{Introduction}
Kazhdan's property (T) is one of the two most important concepts
in analytic group theory (the other being amenability).
The mere existence of an infinite group with property (T) is surprising
and has a wide range of applications. Therefore, it is not so surprising that
any proof of property (T) for any infinite group is necessarily involved.
See the monograph \cite{bhv} for the detailed treatment of property (T).
In this paper, we investigate the noncommutative real algebraic geometric
aspects (cf.\ \cite{nt,cec,schmudgen}) of property (T) and see that
property $\mathrm{(T)}$ is \emph{provable}.
Let $\Gamma$ be a finitely generated group and $\mu$ be a finitely supported
symmetric probability measure on $\Gamma$ whose support generates $\Gamma$.
Then the corresponding Laplacian element $\Delta_\mu$
is the positive element in the real group algebra $\IR[\Gamma]$, defined by
\[
\Delta_\mu=\frac{1}{2}\sum_{x\in\Gamma}\mu(x)(1-x)^*(1-x)=1-\sum_{x\in\Gamma}\mu(x)x.
\]
It is well-known (\cite{valette,bhv}) that the group $\Gamma$ has property (T)
if and only if the Laplacian element $\Delta_\mu$ has a \emph{spectral gap}, i.e.,
there is $\kappa>0$ such that for every orthogonal $\Gamma$-representation $\pi$ on
a real Hilbert space $H$ the spectrum of $\pi(\Delta_\mu)\in\IB(H)$
is contained in $\{0\}\cup[\kappa,+\infty)$.
In turn, this is equivalent to that $\Delta_\mu^2-\kappa\Delta_\mu\geq0$
in the full group $\mathrm{C}^*$-alge\-bra $\mathrm{C}^*\Gamma$.
See Chapter 5 in \cite{bhv} or Chapter 12 in \cite{bo} for the proof of this fact.
The main result of this paper is that the inequality $\Delta_\mu^2-\kappa\Delta_\mu\geq0$
is witnessed in $\IR[\Gamma]$.
\begin{thm*}
Let $\Gamma$, $\mu$, and $\Delta_\mu$ be as above.
Then, the group $\Gamma$ has Kazhdan's property $\mathrm{(T)}$ if and only if
there are a constant $\kappa>0$ and a finite sequence $\xi_1,\ldots,\xi_n$
in $\IR[\Gamma]$ such that $\Delta_\mu^2-\kappa\Delta_\mu=\sum_i\xi_i^*\xi_i$.
Moreover, if $\Gamma$ has property $\mathrm{(T)}$, $\mu$ takes rational values,
and $\kappa>0$ is a rational number such that the spectrum of $\Delta_\mu$
in $\mathrm{C}^*\Gamma$ is contained in $\{0\}\cup(\kappa,+\infty)$,
then
\[
\Delta_\mu^2-\kappa\Delta_\mu\in\{ \sum_{i=1}^n r_i\xi_i^*\xi_i :
  n\in\IN,\,r_i\in\IQ^+,\,\xi_1,\ldots,\xi_n\in\IQ[\Gamma]\}.
\]
\end{thm*}
The present work is motivated by the following two facts. 
The first is that an explicit formula for $\Delta_\mu^2-\kappa\Delta_\mu=\sum_i\xi_i^*\xi_i$
is known in some cases (\cite{bs,zuk}, see Example~\ref{exam} below).
The second is the following result of Shalom.
\begin{cor}[Shalom \cite{shalom}]
Every property $\mathrm{(T)}$ group is a quotient of a finitely presented property $\mathrm{(T)}$ group.
\end{cor}
\begin{proof}
Let $\Gamma$ be a property $\mathrm{(T)}$ group (which is necessarily finitely generated) 
and take a symmetric finite generating subset $S$. Then, by Main Theorem, for the uniform probability 
measure $\mu$ on $S$, there a constant $\kappa>0$ and a finite sequence $\xi_1,\ldots,\xi_n$
in $\IR[\Gamma]$ such that $\Delta_\mu^2-\kappa\Delta_\mu=\sum_i\xi_i^*\xi_i$.
Let $F_S$ be the free group over $S$ and $Q\colon F_S\to\Gamma$ be the corresponding quotient map. 
Let $\tilde{\mu}$ be the uniform probability measure on $S$ which viewed as a subset of $F_S$ 
and $\tilde{\xi}_i\in\IR[F_S]$ be lifts of $\xi_i$ through $Q$. Thus, one has 
$Q(\Delta_{\tilde{\mu}}^2-\kappa\Delta_{\tilde{\mu}}-\sum_i\tilde{\xi}_i^*\tilde{\xi}_i)=0$. 
Therefore for the finite subset 
\[
{\mathcal R}=\{ x^{-1}y \in F_S: x,y\in\mathrm{supp}
 (\Delta_{\tilde{\mu}}^2-\kappa\Delta_{\tilde{\mu}}-\textstyle\sum_i\tilde{\xi}_i^*\tilde{\xi}_i) 
\mbox{ such that }Q(x)=Q(y)\},
\]
the finitely presented group $\langle S\mid{\mathcal R}\rangle$ has property (T), by Main Theorem again, 
and $\Gamma$ is a quotient of $\langle S\mid{\mathcal R}\rangle$.
\end{proof}

Another consequence of Main Theorem is that property (T) is semidecidable, 
a fact which is already observed by Silberman \cite{silberman}.
At this moment, the author is unable to provide new examples of property (T) groups, but only hopes this result
be useful in proving property (T) for some new cases, possibly with an assist of computers.
\subsection*{Acknowledgment}
The author has benefited from the conversations with M. Mimura, H. Sako, Y. Suzuki,
and R. Willet, during the conference ``Metric geometry and analysis'' held in
Kyoto University in December 2013.
\section{Preliminary}
In this section, we review basic facts about convex geometry and their
application to the study of group algebras.
See Chapters II and III of \cite{barvinok} or \cite{cec} for the former, and
\cite{nt,cec,schmudgen} for the latter.
Let $V$ be an $\IR$-vector space and $C\subset V$ be a convex subset.
An element $c\in C$ is called an \emph{algebraic interior} point of $C$
if for every $v\in V$ there is $t\in(0,1]$ such that $(1-t)c+tv \in C$.
Let $\interior(C)$ denote the set of algebraic interior points of $C$.
Notice that for every $c\in\interior(C)$ and $w\in C$, one has
$t c + (1-t)w \in \interior(C)$ for every $t\in(0,1]$.
In particular, $\interior(\interior(C))=\interior(C)$ for every convex subset $C$.
We equip $V$ with a locally convex topology, called the \emph{algebraic topology}
(also known as the finest locally convex topology), by declaring that
any convex set $C$ such that $C=\interior(C)$ is open.
Then, every linear functional on $V$ is continuous and
every linear subspace of $V$ is closed.

Now we consider the case where the $\IR$-vector space $V$ is equipped
with a linear involution $*$ and the subspace $V^{\her}=\{ v\in V : v=v^*\}$
of the hermitian elements has a distinguished cone $V^+\subset V^{\her}$.
A linear functional $\phi\colon V\to\IR$ is said to be positive
if $\phi(v^*)=\phi(v)$ and $\phi(V^+)\subset\IR_{\geq0}$.
We note that every linear functional on $V^{\her}$ uniquely extends to
an hermitian linear functional on $V$.
An element $e\in V^+$ is called an \emph{order unit} if for every $v\in V^{\her}$
there is $R>0$ such that $v+Re \in V^+$. Thus order units are
nothing but algebraic interior points of $V^+$ in $V^{\her}$.
By the Hahn--Banach separation theorem
(also known as the Eidelheit--Kakutani theorem in this context),
one has for every order unit $e\in V^+$ that
\begin{align*}
\overline{V^+} &= \{ v\in V^{\her} :
  \phi(v)\geq 0\mbox{ for all positive linear functionals $\phi$ on $V$}\}\\
 &=\{ v\in V^{\her} : v+\epsilon e \in V^+\mbox{ for every }\epsilon>0\}.
\end{align*}
Moreover, one has $\interior(\overline{V^+})=\interior(V^+)$ whenever $\interior(V^+)\neq\emptyset$.

Next, we recall Positiv\-stellen\-s\"atz for group algebras.
Let $\Gamma$ be a discrete group and $\IR[\Gamma]$ be the real group algebra.
A typical element in $\IR[\Gamma]$ will be denoted by $\xi=\sum_{x\in\Gamma}\xi(x)x$,
where $\xi\colon\Gamma\to\IR$ is a finitely supported function.
The unit of $\Gamma$, as well as the unit of $\IR[\Gamma]$, is simply denoted by $1$.
The $*$-operation is given by $\xi^*(x)=\xi(x^{-1})$,
and the positive cone of $\IR[\Gamma]^{\her}$ is given by the sums of hermitian squares
\[
\Sigma^2\IR[\Gamma]=\{\sum_{i=1}^n \xi_i^*\xi_i :
 n\in\IN,\,\xi_1,\ldots,\xi_n\in\IR[\Gamma]\} \subset \IR[\Gamma]^{\her}.
\]
Then, $1$ is an order unit for $\IR[\Gamma]$ and
the algebra $\IR[\Gamma]$, together with the $*$-positive cone
$\Sigma^2\IR[\Gamma]$, is a semi-pre-$\mathrm{C}^*$-alge\-bra
in the sense of \cite{cec}, and Positiv\-stellen\-s\"atz
(see Corollary 2 in \cite{cec} or Proposition 15 in \cite{schmudgen}) states the following.
\begin{thm}\label{thm:pst}
For $\xi\in \IR[\Gamma]^{\her}$, the following are equivalent.
\begin{enumerate}[$(1)$]
\item
$\pi(\xi)\geq0$ for every orthogonal $\Gamma$-representation $\pi$.
\item
$\xi\in\overline{\Sigma^2\IR[\Gamma]}$,
i.e., for every $\epsilon>0$ one has $\xi+\epsilon1\in\Sigma^2\IR[\Gamma]$.
\end{enumerate}
\end{thm}

Here, by $\pi(\xi)\geq0$ we mean that the self-adjoint operator $\pi(\xi)$
is positive semi-definite.
It follows that for $\Delta_\mu$ as in Introduction, the group
$\Gamma$ has property $\mathrm{(T)}$ if and only if there is
a constant $\kappa>0$ such that for every $\epsilon>0$
one has $\Delta_\mu^2-\kappa\Delta_\mu+\epsilon1\in\Sigma^2\IR[\Gamma]$.
However, this characterization is not very satisfactory, because for a given $\kappa>0$
one has to solve infinitely many equations to verify property (T) of $\Gamma$.
\section{Proof of Main Theorem}
Let $I[\Gamma]=\{ \xi\in\IR[\Gamma] : \sum_x\xi(x)=0\}$ be the \emph{augmented ideal}
of the group algebra $\IR[\Gamma]$. We note that $I[\Gamma]$ is
spanned by $\{ 1-x : x\in\Gamma\}$ and that two natural positive cones
of $I[\Gamma]^{\her}$ coincide:
$I[\Gamma]\cap\Sigma^2\IR[\Gamma]=\Sigma^2I[\Gamma]$ (cf.\ Lemma 4.5 in \cite{nt}).
The reason we look at $I[\Gamma]$ is that the Laplacian element $\Delta_\mu$,
defined in Introduction, belongs to $\Sigma^2 I[\Gamma]$.
In fact, it is an order unit.
\begin{lem}\label{lem:ou}
The Laplacian element $\Delta_\mu$ is an order unit for $I[\Gamma]$. In particular,
\[
\overline{\Sigma^2I[\Gamma]}
 =\{ \xi\in I[\Gamma]^{\her} : \xi+\epsilon\Delta_\mu\in \Sigma^2I[\Gamma]\mbox{ for every }\epsilon>0\}.
\]
\end{lem}
\begin{proof}
For $\xi,\eta\in I[\Gamma]^{\her}$, we write $\xi\le\eta$ if $\eta-\xi\in\Sigma^2I[\Gamma]$.
Since $I[\Gamma]^{\her}$ is spanned by $2-x-x^{-1}=(1-x)^*(1-x)$, $x\in\Gamma$, it suffices to show
\[
\{ x\in\Gamma : \exists R>0\mbox{ such that }(1-x)^*(1-x)\le R\Delta_\mu\}=\Gamma.
\]
It is clear that the left hand side contains the support of $\mu$. Also since
\begin{align*}
(1-xy)^*(1-xy) &= (1-x+x(1-y))^*(1-x+x(1-y))\\
&\le 2(1-x)^*(1-x)+2(1-y)^*(1-y),
\end{align*}
it forms a subgroup and hence is equal to $\Gamma$.
\end{proof}

In general for a subspace $W\subset V$, the inclusion
$\overline{W\cap V^+}\subset W\cap \overline{V^+}$ can be strict.
Here we prove that the equality holds for the case $I[\Gamma]\subset\IR[\Gamma]$.
\begin{lem}\label{lem:cl}
For every positive linear functional $\phi\colon I[\Gamma]\to\IR$, there is a sequence
of positive linear functionals $\phi_n\colon \IR[\Gamma]\to\IR$ such that $\phi_n\to\phi$
pointwise on $I[\Gamma]$. In particular, one has
$\overline{\Sigma^2I[\Gamma]}=I[\Gamma]\cap\overline{\Sigma^2\IR[\Gamma]}$ in $\IR[\Gamma]$.
\end{lem}
\begin{proof}
We use the theory of conditionally negative type functions, for which we refer the reader to
Appendix C in \cite{bhv} or Appendix D in \cite{bo}.
Let a positive linear functional $\phi\colon I[\Gamma]\to\IR$ be given.
We claim that $\psi(x)=\phi(1-x)$ defines a conditionally negative type function on $\Gamma$.
Indeed, for every $\xi\in I[\Gamma]$, one has
\[
\sum_{x,y\in\Gamma}\psi(y^{-1}x)\xi(y)\xi(x)=-\phi(\xi^*\xi)\le0.
\]
(In case $\Gamma$ has property (T), the conditionally negative type function
$\psi$ is bounded and so it is of the form $\|v-\pi(x)v\|$ for some orthogonal
representation $\pi$ and a vector $v$, which implies that $\phi$ extends to
a positive linear functional on $\IR[\Gamma]$.)
It follows from Schoenberg's theorem that the functions
$\phi_t(x):=\exp(-t\psi(x))$ are of positive type for all $t\in\IR_{\geq0}$, i.e.,
they extend to positive linear functionals on $\IR[\Gamma]$.
Since
\[
\lim_{t\to0} t^{-1}\phi_t(1-x) = \lim_{t\to0}\frac{1-\exp(-t\phi(1-x))}{t}=\phi(1-x)
\]
for every $x\in\Gamma$, one has $t^{-1}\phi_t\to\phi$ on $I[\Gamma]$.
The second assertion follows from this.
\end{proof}

Combining the above lemmas with Theorem~\ref{thm:pst}, we arrive at the following.
\begin{prop}
For every $\xi\in I[\Gamma]^{\her}$, the following are equivalent.
\begin{enumerate}[$(1)$]
\item
$\pi(\xi)\geq 0$ for every orthogonal $\Gamma$-representation $\pi$.
\item
$\xi\in\overline{\Sigma^2I[\Gamma]}$, i.e., $\xi+\epsilon\Delta_\mu\in\Sigma^2I[\Gamma]$ for every $\epsilon>0$.
\end{enumerate}
\end{prop}

\begin{proof}[Proof of Main Theorem]
As it is noted in the introduction, the group $\Gamma$ has property (T) if and only if
there is $\kappa'>0$ such that $\Delta_\mu^2-\kappa'\Delta_\mu\in\overline{\Sigma^2\IR[\Gamma]}$.
Thus, the `if' part of the assertion follows. For the other direction,
let us assume that $\Gamma$ has property (T).
Then, Lemmas~\ref{lem:cl} and \ref{lem:ou} imply that
$\Delta_\mu^2-(\kappa'-\epsilon)\Delta_\mu\in\Sigma^2I[\Gamma]$
for every $\epsilon>0$. Taking $\epsilon<\kappa'$ and letting $\kappa=\kappa'-\epsilon$, we are done.
Indeed, this shows that $\Delta_\mu^2-\kappa\Delta_\mu$ is an interior point
of $\Sigma^2I[\Gamma]$ (inside $I[\Gamma]^{\her}$).
Hence, it also belongs to the sub-cone
\[
\{ \sum_{i=1}^n r_i\xi_i^*\xi_i : n\in\IN,\,r_i\in\IR^+,\,\xi_1,\ldots,\xi_n\in\IQ[\Gamma]\cap I[\Gamma]\}
\]
which is dense and has non-empty interior (see the proof of Lemma~\ref{lem:ou}).
Now, if $\mu$ and $\kappa$ are rational, then for a given finite sequence
$\xi_1,\ldots,\xi_n\in\IQ[\Gamma]$, the linear equation
$\Delta_\mu^2-\kappa\Delta_\mu=\sum_{i=1}^n r_i\xi_i^*\xi_i$
in $(r_i)_{i=1}^n$ has a positive real solution if and only if it has
a positive rational solution.
\end{proof}

Examples for which an explicit formula
$\Delta_\mu^2-\kappa\Delta_\mu=\sum_i\xi_i^*\xi_i$ is known
are provided by \cite{bs,zuk}
(see also Theorem 5.6.1 in \cite{bhv} and Theorem 12.1.15 in \cite{bo}).

\begin{exam}\label{exam}
Let $\Gamma$ be a group generated by a finite symmetric set $S$ such that $1\notin S$.
The link is the graph with the vertex set $S$ and the edge set
$E:=\{(x,y) : x^{-1}y\in S\}$. We put on $S$ the probability measure
$\mu(x)=|\{ y\in S :(x,y)\in E\}|/|E|$, and on $E$ the uniform probability measure.
Define $d\colon L^2(S,\mu)\to L^2(E)$ by $(d\xi)(x,y)=\xi(y)-\xi(x)$, and
let $\Lambda=d^*d/2$ be the Laplacian operator on $L^2(S,\mu)$.
If the link graph is connected and the spectrum of $\Lambda$
is contained in $\{0\}\cup[\lambda,+\infty)$, then there are $\xi_i$, $i\in S$, such that
$\Delta_\mu-(2-\lambda^{-1})\Delta_\mu = \sum_i\xi_i^*\xi_i$.
Hence, $\Gamma$ has property (T) if $\lambda>1/2$.
Indeed, for the orthogonal projection $P$ from $L^2(S,\mu)$ onto the
one dimensional subspace of constant functions,
one has $\lambda^{-1}\Lambda+P-I\geq0$
and so there is an operator $T$ on $L^2(S,\mu)$ such that $\lambda^{-1}\Lambda+P-I=T^*T$.
Thus, by writing $A_{x,y}=\ip{A\delta_y,\delta_x}$ for each operator $A$ on $L^2(S,\mu)$, one has
\[
\sum_{x,y\in S^2}(\lambda^{-1}\Lambda_{x,y}+P_{x,y}-I_{x,y}) x^{-1}y
= \sum_{x,y\in S^2}\sum_{i\in S}\mu(i)^{-1}T_{i,x}T_{i,y} x^{-1}y
= \sum_{i\in S}\xi_i^*\xi_i,
\]
where $\xi_i=\mu(i)^{-1/2}\sum_x T_{i,x}x\in\IR[\Gamma]$. But since $\Lambda_{x,x}=\mu(x)$,
$\Lambda_{x,y}=|E|^{-1}$ for $(x,y)\in E$, and $\Lambda_{x,y}=0$ otherwise;
$P_{x,y}=\mu(x)\mu(y)$; and $I_{x,y}=\delta_{x,y}\mu(x)$; the left hand side is
\begin{align*}
\lambda^{-1} & \bigl(1 - \frac{1}{|E|}\sum_{(x,y)\in E}x^{-1}y\bigr)
  + \bigl(\sum_{x,y\in S}\mu(x)\mu(y)x^{-1}y\bigr)-1\\
&= \lambda^{-1}\bigl(1-\sum_x \mu(x)x\bigr)+\bigl(\sum_x \mu(x)x\bigr)^2-1\\
&= \Delta_\mu^2-(2-\lambda^{-1})\Delta_\mu.
\end{align*}
\end{exam}

\newcommand{\aogonek}{\ooalign{a\crcr\hss\raisebox{-1.5ex}[0pt][2pt]{`}}}


\begin{thebibliography}{BHV08}
%
\bibitem[B\'S97]{bs}
W. Ballmann and J. \'Swi{\aogonek}tkowski;
\emph{On $L^2$-cohomology and property $(\mathrm{T})$ for automorphism groups of polyhedral cell complexes.}
Geom. Funct. Anal. \textbf{7} (1997), 615--645.
%
\bibitem[Ba02]{barvinok} A. Barvinok;
\emph{A course in convexity.}
Graduate Studies in Mathematics, 54. American Mathematical Society, Providence, RI, 2002. x+366 pp.
%
\bibitem[BHV08]{bhv} B. Bekka, P. de la Harpe, and A. Valette;
\emph{Kazhdan's property $(\mathrm{T})$}.
New Mathematical Monographs, 11. Cambridge University Press, Cambridge, 2008. xiv+472 pp.
%
\bibitem[BO08]{bo} N. Brown and N. Ozawa;
\emph{$\mathrm{C}^*$-algebras and Finite-Dimensional Approximations.}
Graduate Studies in Mathematics, 88.
American Mathematical Society, Providence, RI, 2008.
%
\bibitem[NT13]{nt} T. Netzer and A. Thom;
\emph{Real closed separation theorems and applications to group algebras.}
Pacific J. Math. \textbf{263} (2013), 435--452.
%
\bibitem[Oz13]{cec} N. Ozawa;
\emph{About the Connes embedding conjecture: algebraic approaches.}
Jpn. J. Math. \textbf{8} (2013), 147--183.
%
%
\bibitem[Sc09]{schmudgen} K. Schm\"udgen;
\emph{Noncommutative real algebraic geometry - some basic concepts and first ideas.}
Emerging applications of algebraic geometry, 325--350, IMA Vol. Math. Appl., 149, Springer, New York, 2009.
%
\bibitem[Sh00]{shalom} Y. Shalom;
\emph{Rigidity of commensurators and irreducible lattices.}
Invent. Math. \textbf{141} (2000), 1--54.
%
\bibitem[Si11]{silberman} L. Silberman; Lectures at ``Metric geometry, algorithms and groups'' at IHP in 2011.\\ 
{\tt http://metric2011.wordpress.com/tag/lior-silbermans-lectures/}
%
\bibitem[Va84]{valette} A. Valette;
\emph{Minimal projections, integrable representations and property $(\mathrm{T})$.}
Arch. Math. (Basel) \textbf{43} (1984), 397--406.
%
\bibitem[\.Zu03]{zuk} A. \.Zuk;
\emph{Property $(\mathrm{T})$ and Kazhdan constants for discrete groups.}
Geom. Funct. Anal. \textbf{13} (2003), 643--670.
\end{thebibliography}
\end{document}